\documentclass[a4paper, reqno, 12pt]{amsart}

\usepackage[margin=25mm, a4paper]{geometry}
\usepackage{color}

\usepackage[T1]{fontenc}
\usepackage{tikz}
\usepackage{mathrsfs}
\usepackage{amsmath, amsfonts, amssymb, amsthm, graphics, graphicx}
\usepackage{amsaddr}
\usepackage{hyperref}
\renewcommand{\\}{\vspace{3mm}}

\newenvironment{edit}{\color{red}}{\color{black}}
\newcommand{\aed}{\begin{edit}}
\newcommand{\zed}{\end{edit}}
\newenvironment{change}{\color{blue}}{\color{black}}
\newcommand{\ach}{\begin{change}}
\newcommand{\zch}{\end{change}}

\providecommand{\aut}{\mathop{\rm Aut \,}\nolimits}
\providecommand{\sym}{\mathop{\rm Sym \,}\nolimits}

\newcommand{\N}{\mathbb{N}}
\newcommand{\Z}{\mathbb{Z}}
\newcommand{\A}{\text{Aut}}
\newcommand{\B}{\mathscr{B}}

\theoremstyle{plain}
\newtheorem{theorem}{Theorem}
\newtheorem*{theorem*}{Theorem}
\newtheorem{corollary}[theorem]{Corollary}
\newtheorem*{corollary*}{Corollary}

\newtheorem{proposition}[theorem]{Proposition}

\theoremstyle{definition}

\newtheorem{remark}[theorem]{Remark}
\newtheorem{example}[theorem]{Example}
\newtheorem*{example*}{Example}
\newtheorem*{remark*}{Remark}

\title{{\bf Bounding the distinguishing number of infinite graphs}}

\author{{\bf Simon M.~Smith\MakeLowercase{$^a$}, Mark E.~Watkins\MakeLowercase{$^b$}}}
\address{$^a$Mathematics Department, NYC College of Technology,\\ City University of New York, NY, USA\\
$^b$Mathematics Department, Syracuse University, Syracuse, NY, USA
\footnote{Email addresses: sismith@citytech.cuny.edu, mewatkin@syr.edu}
}

\begin{document}

\begin{abstract} A group of permutations $G$ of a set $V$ is $k$-{\em distinguishable} if there exists a partition of $V$ into $k$ parts such that only the identity permutation in $G$ fixes setwise all of the cells of the partition.  The least cardinal number $k$ such that $(G,V)$ is $k$-distinguishable is its {\em distinguishing number} $D(G,V)$. In particular, a graph $\Gamma$ is {\em $k$-distinguishable} if its automorphism group $\A(\Gamma)$ satisfies $D(\A(\Gamma),V\Gamma)\leq k$.

Various results in the literature demonstrate that when an infinite graph fails to have some property, then often some finite subgraph is similarly deficient. In this paper we show that whenever an infinite connected graph $\Gamma$ is not $k$-distinguishable (for a given cardinal $k$), then it contains a ball $B$ of finite radius whose distinguishing number is at least $k$. Moreover, this lower bound cannot be sharpened, since for any integer $k \geq 3$ there exists an infinite, locally finite, connected  graph $\Gamma$ that is not $k$-distinguishable but in which every ball of finite radius is $k$-distinguishable. 

In the second half of this paper we show that a large distinguishing number for an imprimitive graph $\Gamma$ is traceable to a high distinguishing number either of a block of imprimitivity or of the induced action of $\A(\Gamma)$ on the corresponding system of imprimitivity. The distinguishing numbers of infinite primitive graphs have been examined in detail in a previous paper by the authors together with T.\,W.\, Tucker.
\end{abstract}

\maketitle

\section{Introduction}

In infinite graph theory, one frequently considers properties that are finitely describable. A theorem in this vein may state that if an infinite graph fails to satisfy a certain finitely describable property, then some finite subgraph is the likely culprit. For example:
\begin{itemize}
\item
	if an infinite graph $\Gamma$ is not $k$-colorable, then there exists a finite subgraph of $\Gamma$ that is not $k$-colorable (N.\,G.\,de Bruijn and P.\,Erd\H{o}s \cite{bruijn_erdos_51}); and
\item
	if a countably infinite graph is not planar, then some (finite) subgraph is homeomorphic to one of the Kuratowski graphs $K_5$ or $K_{3,3}$ and hence is not planar (attributed to P.\,Erd\H{o}s by G.\,Dirac and S.\,Schuster \cite{dirac_schuster}).
\end{itemize}

In this note we consider the property of distinguishability:  a permutation group $G$ acting (faithfully) on a set $V$ (which we often write as a pair $(G,V)$) is $k$-{\em distinguishable} if there exists a partition of $V$ with $k$ cells such that only the identity permutation in $G$ fixes setwise all of the cells of the partition. If $k$ is the minimal cardinal such that the permutation group $(G,V)$ is $k$-distinguishable, then $k$ is the {\em distinguishing number} of $(G,V)$, and we write $D(G,V)=k$.  Applying this notion to graphs, we say that a graph $\Gamma$ is {\em $k$-distinguishable} if its automorphism group $\A(\Gamma)$ satisfies $D(\A(\Gamma),V\Gamma)\leq k$.  (This notion, applied to finite graphs, is originally due to Albertson and Collins \cite{AC1}.)   For brevity, unless some proper subgroup of $\A(\Gamma)$ is being considered, we write simply $D(\Gamma)$ instead of $D(\A(\Gamma),V\Gamma)$.

One might hope that if an infinite graph $\Gamma$ fails to be $k$-distinguishable, then some ``interesting'' substructure ought to bear the blame.  Indeed, this is already known for countably infinite trees: if a countably infinite tree has finite distinguishing number $k$, then some finite subtree also has distinguishing number $k$ (see \cite{watkins_zhou}). In this paper we present two substructures that may be blamed: one is a graph-theoretical substructure and the other is algebraic.

In the first part of this paper we look at combinatorial substructures and demonstrate a class of finite-diameter subgraphs of $\Gamma$ that give a meaningful upper bound for $D(\Gamma)$. In the second part of this paper we look at algebraic substructures, demonstrating a sharp upper bound for $D(\Gamma)$ in terms of the distinguishing number of a block of imprimitivity and the induced action of $\aut(\Gamma)$ on the corresponding system of imprimitivity.\\

For infinite graphs in general, the parameter of distinguishing number is not as well-behaved as parameters such as chromatic number and genus; the distinguishing number of a subgraph of a graph $\Gamma$ is not necessarily less than or equal to, but also may be greater than $D(\Gamma)$. For example, any connected graph with infinite diameter contains finite induced subgraphs with distinguishing number $k$ for all $k \in \mathbb{N}$ (to wit, the null graph on $k$ vertices).
The class of subgraphs of finite diameter that we've selected for consideration are the {\em ball-graphs} $B(x,n)$:  for $n \in \mathbb{N}$,  $B(x,n)$ is the subgraph of $\Gamma$ induced by the vertex set $\{y \in V\Gamma : d(x,y) \leq n\}$. Its  {\em radius} is $n$ and it is {\em centered} at $x$.\\

Suppose that $k-1$ is the largest valence of the vertices of a connected graph $\Gamma$.  If $\Gamma$ is finite, then $D(\Gamma)\leq k$ (see \cite[Theorem 4.2]{collins_trenk}).  When $\Gamma$ is infinite, the sharper bound of $D(\Gamma)\leq k-1$ is obtained (see \cite[Theorem 2.1]{imrich_et_al:distinguishing}).  This easily yields the following.

\begin{proposition}Let $\Gamma$ be a connected graph without $3$-cycles and let  $k$ denote some cardinal. If $\Gamma$ is not $k$-distinguishable, then there exists a vertex $x\in V\Gamma$ such that $B(x, 1)$ has distinguishing number at least $k$. \qed
\end{proposition}

We extend this result considerably.

\begin{theorem} \label{thm_finite_balls}  Let $\Gamma$ be a connected graph and let  $k$ denote some cardinal.  If  $\Gamma$ is not $k$-distinguishable, then, for any vertex $x\in V\Gamma$, all but finitely many ball-graphs centered at $x$ have distinguishing number at least $k$.  
\end{theorem}

\begin{corollary}\label{SharpBound}  If the graph $\Gamma$ of the above theorem is locally finite, then $k$ is a sharp lower bound for the distinguishing number of its ball-graphs.
\end{corollary}

Notice that we are providing here an upper bound for the distinguishing number of $\Gamma$ in terms of the distinguishing number of its finite ball-subgraphs. It is tempting to think it  possible to obtain a more interesting lower bound for infinite graphs $\Gamma$ than $D(\Gamma) \geq 2$ when $\aut(\Gamma)$ is not trivial, but this is impossible. It is easy to construct an example of a connected graph $\Gamma$ in which the distinguishing numbers of the ball-graphs centered at any given vertex of $\Gamma$ are not bounded above, while the whole graph is $2$-distinguishable: consider for example a rooted tree in which, for all $n \in \mathbb{N}$, all the vertices at distance $n$ from the root have valence $n$. \\

The purpose of the second part of this article is to describe the distinguishing number of an  imprimitive graph in terms of its blocks of imprimitivity.  Recall that a transitive group $G \leq \sym(V)$ is {\em primitive} if the only $G$-invariant equivalence relations on $V$ are either trivial or universal.  A graph is {\em primitive} if its automorphism group acts primitively on its vertex set. If $G$ is transitive but not primitive, then it is {\em imprimitive}, and there exists a nontrivial and non-universal $G$-invariant equivalence relation $\rho$ on $V$. Any equivalence class $B$ with respect to $\rho$ is called a {\em block of imprimitivity}, or simply a {\em block}. The set $\B=\{B^g : g \in G\}$ is the set of all equivalence classes of $\rho$, and is called a {\em system of imprimitivity}. The group $G$ naturally induces a transitive group of permutations  on $\B$.  If $H$ is the subgroup of $G$ that fixes every block setwise, then $H \lhd G$ and $G/H$ acts transitively and faithfully on $\B$.  

In \cite[Theorem 1]{akos_seress}, \'{A}.\,Seress showed that all finite primitive permutation groups of degree strictly greater than 32, other than the symmetric and the alternating groups, have distinguishing number 2. It was shown in \cite{sms_twt_mew} that every infinite primitive permutation group with finite {\em suborbits} (orbits of a point-stabilizer) has distinguishing number 2, and thus that the distinguishing number of any nonnull, infinite,  locally-finite, primitive graph is equal to 2.  In the light of these results we investigated imprimitive graphs and concluded that a high distinguishing number for an imprimitive graph $\Gamma$ is accompanied by the property that in any system of imprimitivity $\B$ of $\Gamma$, either:
\begin{enumerate}
\item
	each block $B \in \B$ has a high distinguishing number; or
\item
	the action induced by $\A(\Gamma)$ on the system of imprimitivity has a high distinguishing number.
\end{enumerate}

We determine the reasons for this, obtaining sharp bounds for the distinguishing number of (i) the distinguishing number of any block $B \in \B$ and (ii) any imprimitive permutation group $G$ in terms of the distinguishing number of the induced action of $G$ on any system of imprimitivity $\B$.

For a given cardinal $n$ and a graph $\Lambda$, we denote the disjoint union of $n$ copies of $\Lambda$ by $n \Lambda$.  For each $U\subset V\Gamma$, we let $\langle U\rangle$ denote the subgraph of $\Gamma$ induced by $U$.

The following is the main result of the second part of this article.

\begin{theorem} \label{thm_imprimitive}
Let $\Gamma$ be a vertex-transitive graph such that $(\aut(\Gamma),V\Gamma)$ admits a system of imprimitivity $\B$.  Let $\cong$ denote the equivalence relation on $V\Gamma$ of belonging to the same block.  If the cardinal $n$ denotes the distinguishing number of the action of $G$ on the quotient graph $(\Gamma/\cong)$, then $D(\Gamma)\leq D(n \langle B\rangle)$ for all $B\in\B$.  Furthermore, this bound is sharp.
\end{theorem}

To conclude the article we show that when the cardinal number $n$ is finite, Theorem~\ref{thm_imprimitive} may be deduced (with a little work) from a theorem of Melody Chan \cite[Theorem 2.3]{chan06}.

\section{Distinguishing number and ball-graphs}

We begin by bounding the distinguishing number of a connected graph in terms of the distinguishing number of its ball-subgraphs. Corollary \ref{SharpBound} will follow from Example~\ref{main_example} below.

\begin{proof}[Proof of Theorem~\ref{thm_finite_balls}]  It may be assumed that $\Gamma$ has infinite diameter; otherwise there is nothing to prove.  Since $\Gamma$ is connected, this assumption implies that for all $x\in V\Gamma$ and all $m,n\in\N$, if $m<n$, then $B(x,m)$ is a proper subgraph of $B(x,n)$.

We prove the contrapositive.  Suppose that for some $x \in V\Gamma$ there exists an infinite increasing subsequence $\{n_i\}_{i \in \N}$ from $\N$ such that $D(B(x,n_i)) < k$ for each $i \in \mathbb{N}$.  Let us abbreviate $B(x,n_i)$ by $B(i)$.  Let $X$ be a set (of colors), with $|X|= k$ and fix $c_0 \in X$.  It follows that for each $i\in\N$, there exists a distinguishing coloring $\varphi_i:VB(i)\to X$ with the property that $\varphi_i(y)=c_0$ if and only if $y=x$.

We now construct a coloring $\psi:V\Gamma\to X$ with the property that $\psi(y)=c_0$ if and only if $y=x$ and prove by induction on $i$ that $\psi$ is $k$-distinguishing on $B(i)\setminus B(i-1)$ for all $i\in\N$.  From this it will follow that $D(\Gamma)\leq k$.

We begin by setting $\psi_1=\varphi_1$ and remarking that $\psi_1$ is a distinguishing coloring of $B(1)$ with at most $k$ colors that assigns to $y \in VB(1)$ the color $c_0$ if and only if $y=x$.  For $j\geq2$ we define the $k$-coloring $\psi_j: VB(j)\to X$ by

\[\psi_j(y) = \begin{cases} \psi_{j-1}(y) \quad & \text{if $y \in VB(j-1)$,}\\ \varphi_j(y) \quad & \text{if $y\in VB(j)\setminus VB(j-1)$.} \end{cases}\]
Our induction hypothesis is that for all $i<j$, $\psi_i$ is a distinguishing coloring of $B(i)$ that agrees with $\psi_{i-1}$ on $VB(i-1)$.  We claim that $\psi_j$ is a distinguishing coloring of $B(j)$ and is an extension of $\psi_{j-1}$.  For some $g \in \aut(B(j))$, suppose that $\psi_j(y^g) = \psi_j(y)$ for all $y \in VB(j)$. Since $\psi_j(y) = c_0$ if and only if $y = x$, we have that $g$ fixes $x$ and therefore $g$ fixes $VB(j-1)$ setwise. Since $\psi_{j-1}$ is a distinguishing coloring of $B(j-1)$ while $\psi_j$ and $\psi_{j-1}$ agree on $VB(j-1)$, $\psi_j$ restricted to $B(j-1)$ is a distinguishing coloring of $B(j-1)$; hence $g$ fixes $VB(j-1)$ pointwise.  But $g$ also fixes $VB(j) \setminus VB(j-1)$ setwise.  Hence for all $y\in VB(j) \setminus VB(j-1)$, we have $\varphi_j(y)=\psi_j(y)=\psi_j(y^g)=\varphi_j(y^g)$. We have shown that for all $y \in B(j)$ we have $\varphi_j(y) = \varphi_j(y^g)$, which implies $y=y^g$ because $\varphi_j$ is a distinguishing coloring of $B(j)$.  Hence $\psi_j$ is a distinguishing coloring of $B(j)$ that agrees with $\psi_i$ on $B(i)$ whenever $i\leq j$.

Define a function $\psi : V\Gamma \rightarrow X$ as $\psi(y) = \psi_i(y)$ whenever $y \in VB(i) $. The argument of the preceding paragraph implies that $\psi$ is well-defined.  We claim that $\psi$ is a distinguishing coloring of $\Gamma$. For suppose that $\psi(y^g) = \psi(y)$ for some $g \in \aut(\Gamma)$ and all $y \in V\Gamma$.  Then $g$ fixes $x$ and therefore $g$ fixes setwise every set $VB(i)$. Moreover, for all $i \in \mathbb{N}$ and for all $y \in VB(i)$, the function $\psi_i$ is a distinguishing coloring of $B(i)$; since $\psi_i(y^g) = \psi(y^g) = \psi(y) = \psi_i(y)$, it follows that $y^g = y$. Hence $\psi$ is a distinguishing coloring of $\Gamma$, and so $D(\Gamma) \leq|X|=k$.  
\end{proof}

\begin{example} \label{main_example} This example demonstrates the sharpness of the lower bound in the case of locally finite graphs.  For any given integer $k \geq 3$, we construct an infinite, locally finite graph $\Gamma$ with the following two properties:
\begin{enumerate}
\item
	$D(\Gamma)=k+1$; and
\item For all $x \in V\Gamma$, all but finitely many ball-graphs centered at $x$ have distinguishing number $k$.
\end{enumerate}

Let $A_0$ be the complete graph $K_k$ on $k$ vertices; let $A_1$ be the complete graph $K_{k(k-1)}$ minus a $1$-factor, and let $A_2$ be the complement of $A_0$ (i.e., the null graph of order $k$).  Write $[n] := n\pmod{3}$. Let $\Gamma$ be the infinite, locally finite graph with vertex set $V\Gamma = \bigcup_{n \in \Z} \left ( VA_{[n]}\times \{n\} \right )$, in which two vertices $( x,m), (y,n) \in V\Gamma$ are adjacent if and only if 
\begin{enumerate}
\item
	$n=m$ and $x$ and $y$ are adjacent in $A_{[n]}$; or
\item
	$|n-m|=1$.
\end{enumerate}
Thus $\Gamma$ is a strip (i.e., a 2-ended graph admitting a translation; see \cite{JW}).   Intuitively, $\Gamma$ has the following form:
\[\cdots-A_2 - A_0-A_1-A_2-A_0-A_1-\cdots\]
in which each vertex in any copy of $A_{[n]}$ is adjacent to every vertex in its adjacent copies of $A_{[n-1]}$ and $A_{[n+1]}$.  For $n \in\Z$, let $H_n$ be the subgraph of $\Gamma$ induced by $VA_{[n]}\times\{n\}$ (so $H_n \cong A_{[n]}$), and let $\mathcal{H}:=\{H_n : n \in \mathbb{Z}\}$.
\medskip

Let us first examine $\aut(\Gamma)$.  We remark that each of the subgraphs $H_n$ is a vertex-transitive graph.  One easily verifies that for all $m\in\N$ the valences (in $\Gamma$) of vertices in $H_{3m},\ H_{3m+1}$, and $H_{3m+2}$ are, respectively, $k^2+k-1$, $k^2+k-2$, and  $k^2$.  Since $k\geq3$, these three integers are distinct, and so the orbit of any vertex in $VH_n$ is the set $\bigcup\{VH_m:m\equiv n\pmod3\}$.

Fix $g \in \aut(\Gamma)$.  Since $H_0 \cong A_0$ is connected, it must therefore hold  that $H_0^g=H_{3m}$ for some $m\in\Z$.  Since $A_1$ is also connected, and since every vertex in $H_0$ is adjacent to every vertex in $H_1$, it follows that $H_1^g=H_{3m+1}$.  Since the $k$ vertices of $H_{3m+2}$ must be the images of the remaining $k$ neighbors of the vertices of $H_1$, we have $H_2^g=H_{3m+2}$.  Proceeding in this manner through both positive and negative subscripts, one shows inductively that $g$ satisfies $H_n^g=H_{3m+n}$ for all $n\in\Z$.  If $m\neq0$, then $g$ is a translation.

To prove that  that $\Gamma$ is not $k$-distinguishable, suppose that $\varphi:V\Gamma\to\{1, 2, \ldots, k\}$ is a distinguishing coloring of $\Gamma$.  Using each of the $\binom{k}2$ (unordered) pairs of colors for each of the $\binom{k}2$ pairs of nonadjacent vertices of $A_1$, we have $D(A_1)=k$.  Thus $D(H_n)=k$ for all $n\in\Z$.  Specifically, there exists a distinguishing $k$-coloring of each subgraph $H_n$, and it is unique up to a permutation of the colors.  Moreover, for each $n\in\Z$, there exists an automorphism $h_n\in\A(H_n)$ such that, for all $(x,n)\in V\Gamma$, we have $\varphi(x,n)=\varphi\left((x,n+3)^{h_{n+3}}\right)$.  But then all $k$ color classes determined by $\varphi$ are preserved by the translation $(x,n)\mapsto(x,n+3)^{h_{n+3}}$.

We next show that $\Gamma$ is $(k+1)$-distinguishable.  For our palette of colors, we now use  $X := \{0, 1, 2, \ldots,k\}$.  By the previous argument, we know that we may let $\varphi_0 : VH_0 \rightarrow X \setminus \{k\}$ be a distinguishing $k$-coloring of $H_0$, and for any integer $n \neq 0$ let $\varphi_n : VH_n \rightarrow X \setminus \{0\}$ be a distinguishing $k$-coloring of $H_n$.  The unique vertex of $H_0$ colored 0 must therefore be fixed by any automorphism $g\in\aut(\Gamma)$ that preserves the color classes of $\varphi$, and so $g$ fixes setwise each subgraph $H_n$.  But $g$ preserves the distinguishing coloring $\varphi_n$ of $H_n$ for every $n \in \Z$. Hence $g$ is the identity automorphism.

It remains only to show that for any given $(x, n) \in V\Gamma$ and integer $m\geq3$, the ball-graph $B:=B\left ((x,n), m \right )$ has distinguishing number at most $k$.  Clearly
$$VB = \bigcup_{i=n-m}^{n+m} VH_{i}.$$
If $|n-i|<m$, then the valence of a vertex in $H_i$ is the same in both $B$ and $\Gamma$.  If $|n-i|=m$, then the valence of a vertex in $H_i$ is one of the five smaller values:  $k^2-1$, $k^2-2,\ k^2-k,\ 2k-1,\  k$.  By an inductive argument similar to the one above, it follows that every automorphism of $B$ fixes $H_i$ setwise whenever $|n-i|\leq m$.  Since $D(H_i)=k$ it must hold that $D(B) \leq k$.

We remark that one can find a $k$-subset of the vertices of $H_{3m-1}$ whose union with $H_{3m}$ is isomorphic to $K_{2k}$, but the resulting graph is not a ball-graph.   
\end{example}

\section{Distinguishing number and imprimitivity}

In this section, given an arbitrary infinite set $V$, we obtain an upper bound for the distinguishing number for a group $G$ of permutations acting imprimitively on $V$.  We then apply this bound to the group of automorphisms of a locally finite graph in order to demonstrate that our bound is sharp.  These results complement those of Section 3 of \cite{sms_twt_mew}, which concern primitive group actions.  Since imprimitive groups can be embedded in wreath products in a natural way, we first present for completeness a definition and notation for the wreath product of two permutation groups.  (See, for example, \cite[pp 67--72]{bat_mac_mol_neu}.)

Let $H \leq \sym(A)$ and $K \leq \sym(B)$.  Then the {\em wreath product} $H \wr K$ is defined to be the semidirect product $\text{Fun}(B, H) \rtimes K$, where $\text{Fun}(B, H)$ is the group of functions from $B$ to $H$. The wreath product $H \wr K$ has a faithful action (called the imprimitive action) on $A \times B$, defined as follows: for all $(a,b) \in A \times B$, $f \in \text{Fun}(B, H)$, and $k \in K$,
\begin{equation}\label{defn_wreath}
(a,b)^{(f, k)} := (a^{f(b)}, b^k).
\end{equation}

\begin{theorem} \label{thm:wr} Let $G \leq \sym(V)$ be a transitive group of permutations, let $\B$ be a system of imprimitivity induced by $G$, and let $A\in\B$. Let $H$ be the subgroup of $\sym(A)$ induced by the setwise stabilizer $G_{\{A\}}$.  If $X$ is any set such that  $|X| =D(G, \B)$, then
\[D(G, V) \leq D(H \wr \sym(X), A \times X).\]
\end{theorem}

\begin{proof}
Fix some block $A \in \B$. Let $\chi: \B \rightarrow X$ be a distinguishing coloring of $(G, \B)$, and let $\psi: A \times X \rightarrow Y$ be a distinguishing coloring of $(H \wr \sym(X), A \times X)$, where $Y$ is some sufficiently large set of colors.  Thus $D(H \wr \sym(X), A \times X)\leq|Y|$.

Since $(G, \B)$ is transitive, for each $B\in\B$, there exists $g_B\in G$ such that $B^{g_B}=A$.  This defines an injection $f:\B\to G$ given by $B\mapsto g_B$; that is,
\begin{equation*}
B^{f(B)}=A\ {\text{for each }}B\in\B.
\end{equation*}
We now define a coloring  $\phi: V \rightarrow Y$ as follows: for each $v \in V$, if $B$ is the block in $\B$ containing $v$, then
\begin{equation*}
\phi(v) := \psi \left ( \left ( v^{f(B)}, \chi(B) \right ) \right )
\end{equation*}
It remains only to show that $\phi$ describes a distinguishing coloring of $(G, V)$, for this will imply that, for any set $Y$, if $D(H \wr \sym(X), A \times X)\leq|Y|$, then $D(G,V)\leq|Y|$.

Suppose that some permutation $g\in G$ preserves all the color classes of $\phi$ in $V$. If $x \in A$ and $B \in\B$, then we have
\[\phi(x^{f(B)^{-1} g}) = \phi(x^{f(B)^{-1}}) = \psi \left ( \left ( x^{f(B)^{-1} f(B)}, \chi(B) \right ) \right ) = \psi \left ( \left ( x, \chi(B) \right ) \right ).\]
However, we also have that
\[\phi(x^{ f(B)^{-1} g}) = \psi \left ( \left ( x^{ f(B)^{-1} g \, f(B^g)}, \chi(B^g) \right ) \right ).\]
Hence, for all $x \in A$ and $B \in \B$, and for all $g \in G$ that preserve the coloring function $\phi$, we have
\begin{equation}\label{eq:b}
\psi \left ( \left ( x, \chi(B) \right ) \right ) = \psi \left ( \left ( x^{f(B)^{-1} g \,  f(B^g)}, \chi(B^g) \right ) \right ).
\end{equation}

Fix some $g \in G$ that preserves the coloring $\phi$ of $V$. We now show that $g$ must fix $V$ pointwise, from which it follows that $\phi$ is a distinguishing coloring of $(G,V)$. Fix $B \in \B$ and note that ${f(B)^{-1} g \, f(B^g)} \in G_{\{A\}}$. Let $h \in H$ be the permutation of $A$ induced by ${f(B)^{-1} g \, f(B^g)}$. Let $\sigma \in \sym(X)$ be the permutation of $X$ that interchanges the colors $\chi(B)$ and $\chi(B^g)$ and fixes every other element of $X$; thus either (i) $\sigma$ is is a transposition or (ii) $\sigma=1_X$.

Let us define $\theta : X \rightarrow H$ by
\[\theta(i) = 
\begin{cases}
h & \text{ \ if $i = \chi(B)$;}\\
h^{-1} & \text{ \ if $i = \chi(B^g)$;}\\
1_H & \text{ \ otherwise.}
\end{cases}
\]
Thus $(\theta, \sigma) \in H \wr \sym(X)$, and we must apply Equation (\ref{defn_wreath}) to evaluate $(x, \chi(B))^{(\theta, \sigma)}\in A \times X$ in each of the two cases.

In Case (i), where $\chi(B) \neq \chi(B^g)$, we have by Equation (\ref{defn_wreath}) that $(x, \chi(B))^{(\theta, \sigma)}=(x^{\theta(\chi(B))},\chi(B)^\sigma)=(x^h, \chi(B^g))$ and $(x^h, \chi(B^g))^{(\theta, \sigma)} = (x, \chi(B))$ for all $x\in A$, while all other elements of $A \times X$ remain fixed by $(\theta, \sigma)$. Thus, by Equation (\ref{eq:b}), the permutation $(\theta, \sigma)$ preserves the color classes of $\psi$ on $A\times X$ and is therefore the identity. Since this contradicts the assumption that $\chi(B) \not = \chi(B^g)$, Case (i) is not possible.

So, we must have Case (ii), where $\chi(B) = \chi(B^g)$. We now define
\[\theta(i) = 
\begin{cases}
h & \text{ \ if $i = \chi(B)$;}\\
1_H & \text{ \ otherwise}.
\end{cases}
\]
Applying Equation (\ref{defn_wreath}) again (and noting that $\sigma$ is trivial), we have  $(x, \chi(B))^{(\theta, \sigma)} = (x^h, \chi(B)) = (x^h, \chi(B^g))$ for all $x \in A$; every other element of $A \times X$ is fixed by $(\theta, \sigma)$. Thus by Equation (\ref{eq:b}) the permutation $(\theta, \sigma)$ preserves the coloring $\psi$, and is therefore the identity on $A\times X$. In particular, $h=1_A$.

Since $B \in \B$ was chosen arbitrarily, we have shown that for all $B \in \B$,
\[\chi(B) = \chi(B^g)\]
and
\[f(B)^{-1} g f(B) \in G_{(A)},\]
where $G_{(A)}$ here denotes the pointwise stabilizer in $G$ of $A$. Thus the action of $g$ on $\B$ preserves $\chi$, and so $g$ fixes each block in $\B$ setwise. Furthermore, if $y \in V$ then there exists some $B \in \B$ and $x \in A$ such that $y = x^{f(B)^{-1}} = \left ( x^{f(B)^{-1} g f(B)} \right )^{f(B)^{-1}} = x^{f(B)^{-1} g} = y^g$. Hence $g$ fixes $V$ pointwise, and $\phi$ is a distinguishing coloring of $(G, V)$.
\end{proof}

Using Theorem~\ref{thm:wr} it is not difficult to obtain a proof of Theorem~\ref{thm_imprimitive}. The key to the proof is the well-known observation that imprimitive permutation groups can be embedded inside wreath products.

\begin{proof}[Proof of Theorem~\ref{thm_imprimitive}]
Let $H$ be the subgroup of $G:=\aut(\langle B \rangle)$ induced by the setwise stabilizer $G_{\{B\}}$, and let  $N$ be a set such that $|N|=n$.  Represent  the vertex set of $n \langle B \rangle$ as $B\times N=\{(x, \nu) : x \in B;\  \nu \in N\}$, where we understand that for any given $\nu_0\in N$, the set $\{(x, \nu_0) : x \in B\}$ spans a copy of $\langle B\rangle$.

Since $G \wr \sym(N)$ acts as a group of permutations on $B\times N$, we have $G \wr \sym(N) \leq \sym(B \times N)$, and since $G \wr \sym(N)$ preserves the edge structure of $n \langle B \rangle$, we have $G \wr \sym(N) \leq \aut(n \langle B \rangle)$.  We now apply Theorem~\ref{thm:wr}, giving
$$D(\Gamma) \leq D(H\wr \sym(N), B \times N) \leq D(\aut(\langle B \rangle) \wr \sym(N), B \times N) \leq D(n \langle B \rangle).$$ 
This bound is sharp, since given a cardinal $n$ and a connected graph $B$, one could choose $\Gamma$ to be the graph $nB$.
\end{proof}

\begin{remark} The bound in Theorem~\ref{thm_imprimitive} is sharp even for connected graphs, since $n \langle B \rangle$ and its complement have the same distinguishing number.
\end{remark}

We are now able to bound the distinguishing number of an imprimitive graph in a simple way.

\begin{corollary} \label{thm_bounds} Under the hypothesis of Theorem~\ref{thm_imprimitive}, if $n:=D(\Gamma / \cong)$ and  $k := D(\langle B \rangle)$ are finite, then
\[D(\Gamma) \leq kn^{1/k}+1.\]
\end{corollary}

\begin{proof} Let $m$ be the integer satisfying $kn^{1/k} \leq m < kn^{1/k} +1$.  Since $m\geq k$, we have $\binom{m}{k}\geq(m/k)^k\geq n$.
But if $m$ satisfies $\binom{m}{k}\geq n$, then $D(n\langle B\rangle)\leq m$, because a different $k$-set of colors may be used for each copy of $\langle B\rangle$.  Hence by Theorem \ref{thm_imprimitive}, $D(\Gamma)\leq D(n\langle B\rangle)\leq m\leq kn^{1/k}+1$.
\end{proof}

For a permutation group $(H, A)$ let $n_r(H, A)$ be the number of distinct distinguishing $r$-colorings of $(H, A)$.  For $S \subseteq \N$ let
\[{\min}^* S := \begin{cases} \min S \quad \text{if $S \not = \emptyset$; and} \\ \aleph_0 \quad \text{if $S = \emptyset$.} \end{cases}\]
In the Introduction of this article, we referred to a result of Melody Chan:  
\begin{proposition}[M.\,Chan {\cite[Theorem 2.3]{chan06}}] \label{chan} If $(H, A)$ and $(K, B)$ are permutation groups and $D(K, B)$ is finite, then
\[D(H \wr K, A \times B) = {\min}^* \left \{ r \in \N : n_r(H, A) \geq |H|\cdot D(K, B) \right \}.\]
\end{proposition}

We conclude by showing how Chan's result implies Theorems \ref{thm_imprimitive} and \ref{thm:wr} in the case if finite distinguishing numbers.\medskip

\begin{proof}Suppose that $(G, V)$ is a transitive permutation group that induces a system of imprimitivity $\B$. Let $G^{\B}$ be the subgroup of $\sym(\B)$ induced by the action of $G$ on $\B$.  Suppose that 
$D(G^{\B}, \B)=n$, where $n$ is a positive integer.  Let $X$ denote an $n$-set of colors.  Let $A \in \B$, and let $H$ be the subgroup of $\sym(A)$ induced by the setwise stabilizer $G_{\{A\}}$,
Observe (by \cite[Theorem 8.5]{bat_mac_mol_neu}) that $(G, V)$ is permutation-isomorphic to a subgroup of $(H \wr G^{\B}, A \times \B)$. Hence,
$D(G, V) \leq D(H \wr G^{\B}, A \times \B)$
and
$D(G^{\B}, \B) = D(\sym(X), X)$.
Since $n$ is finite,  Proposition~\ref{chan} yields:
\begin{align*} D(H \wr G^{\B}, A \times \B)
&= {\min}^* \left \{ r \in \N : n_r(H, A) \geq |H|\cdot D(G^{\B}, \B) \right \} \\
&= {\min}^* \left \{ r \in \N : n_r(H, A)\geq |H|\cdot D(\sym(X), X) \right \} \\
&= D(H \wr \sym(X), A \times X).
\end{align*}
We have thus deduced the statement of Theorem~\ref{thm:wr} in the case where $n$ is finite.  Recall that Theorem \ref{thm:wr} is used in the proof of Theorem~\ref{thm_imprimitive}.
\end{proof}

\vspace{1cm}
\noindent Acknowledgements: Much of this work was completed while the first author was a Philip T Church Postdoctoral Fellow at Syracuse University. The second author was partially supported by a grant from the Simons Foundation (\#209803 to Mark E.\,Watkins). 

\end{document}